\documentclass[12pt]{amsart}
\usepackage{amscd, amsfonts, amssymb}
\usepackage{fullpage}
\usepackage{latexsym}
\usepackage{verbatim}
\usepackage{enumerate}
\usepackage[all,cmtip]{xy}
\usepackage[colorlinks=true]{hyperref}

\numberwithin{equation}{section}

\newtheorem{thm}[equation]{Theorem}

\newtheorem{cor}[equation]{Corollary}
\newtheorem{prop}[equation]{Proposition}

\theoremstyle{definition}

\theoremstyle{remark}
\newtheorem{rem}[equation]{Remark}
\theoremstyle{remark}

\newcommand{\ovl}{\overline}


\date{April 22, 2024}

\title[Powers of commutators in linear algebraic groups]
{Powers of commutators in linear algebraic groups}

\author[B.\ Martin]{Benjamin Martin}
\address
{Department of Mathematics,
University of Aberdeen,
King's College,
Fraser Noble Building,
Aberdeen AB24 3UE,
United Kingdom}
\email{B.Martin@abdn.ac.uk}

\begin{document}

\begin{abstract}
 Let ${\mathcal G}$ be a linear algebraic group over $k$, where $k$ is an algebraically closed field, a pseudo-finite field or the valuation ring of a nonarchimedean local field.  Let $G= {\mathcal G}(k)$.  We prove that if $\gamma\in G$ such that $\gamma$ is a commutator and $\delta\in G$ such that $\langle \delta\rangle= \langle \gamma\rangle$ then $\delta$ is a commutator.  This generalises a result of Honda for finite groups.  Our proof uses the Lefschetz Principle from first-order model theory.
\end{abstract}

\maketitle

\section{The main result}

Let $G$ be a group.  We say that $G$ {\em has the Honda property}\footnote{This terminology was suggested to me by Hendrik Lenstra.} if for any $\gamma\in G$ such that $\gamma$ is a commutator and for any $\delta\in G$ such that $\langle \delta\rangle= \langle \gamma\rangle$, $\delta$ is also a commutator.  If $\gamma$ has infinite order then the only generators of $\langle \gamma\rangle$ are $\delta= \gamma^{\pm 1}$, so the condition above is only of interest when $\gamma$ has finite order.  The following result was proved by Honda in 1953 \cite{honda}.

\begin{thm}
\label{thm:finite}
 Any finite group has the Honda property.
\end{thm}

\noindent Honda's original proof is character-theoretic.  Recently Lenstra has given a short and elegant proof that avoids character theory completely \cite{lenstra}.

It is natural to ask which other groups have the Honda property.  Pride has given an example of a one-relator group with torsion which does not have the Honda property \cite[Result (C), p488]{pride}.  In this note we extend Theorem~\ref{thm:finite} to certain linear algebraic groups.  By a linear algebraic group over a ring $k$ we mean a smooth closed subgroup scheme of ${\rm GL}_n$ for some $n\in {\mathbb N}$; if $k$ is a field then these are just linear algebraic groups in the usual sense \cite{borel}, and they correspond to the affine algebraic groups of finite type by \cite[Cor.\ 4.10]{milne}.

\begin{thm}
\label{thm:algebraic}
 Let ${\mathcal G}$ be a linear algebraic group over $k$ and let $G= {\mathcal G}(k)$, where $k$ is one of the following.
 
 (a) An algebraically closed field.
 
 (b) A pseudo-finite field.
 
 (c) The valuation ring of a nonarchimedean local field.
 
 \noindent Then $G$ has the Honda property.
\end{thm}

\begin{rem}
\label{rem:strong}
 Let us say a group $G$ {\em has the strong Honda property} if for any $\alpha, \beta\in G$ and any $\delta\in G$ such that $\langle \delta\rangle= \langle [\alpha, \beta]\rangle$, there exist $\sigma, \tau\in \langle \alpha, \beta\rangle$ such that $\delta= [\sigma, \tau]$.  Clearly if $G$ has the strong Honda property then $G$ has the Honda property and any subgroup of $G$ has the strong Honda property.  Theorem~\ref{thm:finite} implies that any locally finite group has the strong Honda property.  In particular, if ${\mathcal G}$ is a linear algebraic group over $\ovl{{\mathbb F}_p}$ then ${\mathcal G}(\ovl{{\mathbb F}_p})$ is locally finite, so ${\mathcal G}(\ovl{{\mathbb F}_p})$ has the strong Honda property.  (Here $\ovl{{\mathbb F}_p}$ denotes the algebraic closure of the field with $p$ elements.)
 
 On the other hand, none of ${\rm (P)GL}_n({\mathbb C})$, ${\rm (P)GL}_n({\mathbb R})$, ${\rm (P)SL}_n({\mathbb C})$ and ${\rm (P)SL}_n({\mathbb R})$ has the strong Honda property if $n\geq 2$, so Theorem~\ref{thm:algebraic}(a) fails if we replace the Honda property with the strong Honda property.  To see this, consider the group $\Gamma_s= \langle a,b\,|\,[a,b]^s\rangle$.  Pride ({\em loc.\ cit.}) shows that if $0< t< s$ then $[a,b]^t$ is not a commutator in $\Gamma_s$ unless $t\equiv \pm 1\! \mod\! s$, so $\Gamma_s$ does not have the Honda property if $s> 6$ (choose $2\leq t\leq s- 2$ such that $t$ is coprime to $s$).  The group $\Gamma_s$ is Fuchsian (cf.\ \cite[Sec.\ 1]{LMR}), so it is a subgroup of ${\rm PSL}_2({\mathbb R})$.  Hence ${\rm PSL}_2({\mathbb R})$ does not have the strong Honda property, so ${\rm PSL}_2({\mathbb C})$, ${\rm PGL}_2({\mathbb R})$ and ${\rm PGL}_2({\mathbb C})$ also do not have the strong Honda property.  Now choose lifts $\widehat{a}$ (resp., $\widehat{b}$) of $a$ (resp., $b$) to elements of ${\rm SL}_2({\mathbb R})$, and let $\widehat{\Gamma}_s= \langle \widehat{a}, \widehat{b}\rangle$.  The element $[\widehat{a}, \widehat{b}]$ has order either $s$ or $2s$, and $[\widehat{a}, \widehat{b}]^t$ is not a commutator in $\widehat{\Gamma}_s$ if $0< t< s$ and $t\not\equiv \pm 1\mod\ s$.  Hence ${\rm SL}_2({\mathbb R})$ does not have the strong Honda property (choose $s> 6$ and choose $2\leq t\leq s- 2$ such that $t$ is coprime to $2s$), so ${\rm GL}_2({\mathbb R})$, ${\rm SL}_2({\mathbb C})$ and ${\rm GL}_2({\mathbb C})$ also do not have the strong Honda property.  Since we may embed ${\rm SL}_2({\mathbb R})$ as a subgroup of ${\rm (P)GL}_n({\mathbb C})$, ${\rm (P)GL}_n({\mathbb R})$, ${\rm (P)SL}_n({\mathbb C})$ and ${\rm (P)SL}_n({\mathbb R})$ for any $n\geq 3$, the assertion above follows.
\end{rem}

The proof of Theorem~\ref{thm:algebraic}(a) uses the Lefschetz Principle from first-order model theory.  The idea is very simple.  Let ${\mathcal G}$ be a linear algebraic group over an algebraically closed field $k$.  For the moment, assume ${\mathcal G}$ is defined over the prime field $F$.  Set $G= {\mathcal G}(k)$.  We can find an $F$-embedding of ${\mathcal G}$ as a closed subgroup of ${\rm SL}_n$ for some $n$: so $G$ is defined as a subset of ${\rm SL}_n(k)$ by polynomials over $F$ in $n^2$ variables and the group operations on $G$ are given by polynomial maps over $F$.  If $\gamma, \delta\in G$ then $\langle \delta\rangle= \langle \gamma\rangle$ if and only if $\delta= \gamma^s$ and $\gamma= \delta^t$ for some $s,t\in {\mathbb Z}$.  For fixed $s$ and $t$, we observe that the condition

\begin{equation*}
 (*)_{s,t} \ \ \ \mbox{for all $\gamma, \delta\in G$, if $\gamma$ is a commutator and $\delta= \gamma^s$ and $\gamma= \delta^t$ then $\delta$ is a commutator}
\end{equation*}

\bigskip
\noindent is given by a first-order sentence in the language of rings involving the $n^2$ variables of the ambient affine space.  Now $(*)_{s,t}$ is true if $k= \ovl{{\mathbb F}_p}$ for any prime $p$, by Remark~\ref{rem:strong}.  It follows from the Lefschetz Principle that $(*)_{s,t}$ is true for every algebraically closed field $k$, including the characteristic 0 case.  But $s$ and $t$ were arbitrary integers, so the result follows.  On the other hand, we see from Remark~\ref{rem:strong} that the analogous argument fails for the strong Honda property, so the strong Honda property cannot be expressed in a first-order way.

If $G$ is not defined over the prime field $F$ then the argument above needs modification.  The trouble is that the polynomials that define $G$ as a closed subgroup of ${\rm SL}_n(k)$ may involve parameters from $k$, so $(*)_{s,t}$ may fail to translate into an honest sentence.  We get around this by using a trick from \cite{MST2} (cf.\ also the discussion after Theorem 1.1 of \cite{MST1}): first we replace these parameters with variables, then we quantify over all possible values of these variables.  This amounts to quantifying over all linear algebraic groups of bounded complexity in an appropriate sense.  In \cite[Sec.\ 3.2]{MST2} this idea is formulated using the language of Hopf algebras; here we give a more concrete description.

We present the details in Section~\ref{sec:Lefschetz} below.  The proof of Theorem~\ref{thm:algebraic}(b) follows from a closely related argument.  We prove Theorem~\ref{thm:algebraic}(c) in Section~\ref{sec:profinite}.

\begin{rem}
 J. Wilson gives some other first-order properties of groups that hold for finite groups but not for arbitrary groups \cite{wilson}; see \cite{wilson06}, \cite{CW}.  One can use the methods of this paper to prove that these properties hold for linear algebraic groups over an algebraically closed or pseudo-finite field.  Likewise, Theorem~\ref{thm:algebraic} holds for definable groups in the sense of \cite[introduction]{PPS} over an algebraically closed or pseudo-finite field $k$; in particular, this includes group schemes of finite type over such $k$.  I'm grateful to Lenstra and Tiersma for these observations.
\end{rem}

\section{Model theory and the Lefschetz Principle}
\label{sec:Lefschetz}

We give a brief review of the model theory we need to prove Theorem~\ref{thm:algebraic}.  For more details, see \cite{marker}.  We work in the language of rings, which consists of two binary function symbols $+$ and $\cdot$, and two constant symbols 0 and 1.  A {\em formula} is a well-formed expression involving $+$, $\cdot$, $=$, 0 and 1, the symbols $\vee$ (or), $\wedge$ (and), $\rightarrow$ (implies), $\neg$ (not), some variables and the quantifiers $\forall$ and $\exists$.  For instance, $(\exists Y)\,Y^2= X$ and $(\forall X) (\exists Y) (\exists Z)\,X= Y^2+ Z^2+ 2$ and $X\neq 0 \rightarrow X^2= X$ are formulas.  The variable $X$ in the first formula is {\em free} because it is not attached to a quantifier, while the variable $Y$ is {\em bound}.  A {\em sentence} is a formula with no free variables; the second formula above is a sentence.  For any given ring $k$, a sentence is either true or false: for instance, the sentence $(\forall X) (\exists Y) (\exists Z)\,X= Y^2+ Z^2+ 2$ is true when $k= {\mathbb C}$ and false when $k= {\mathbb R}$.  If a formula involves one or more free variables then it doesn't make sense to ask whether it is true or false for a particular ring $k$; but if $\psi(X_1,\ldots, X_m)$ is a formula in free variables $X_1, \ldots, X_m$, $k$ is a ring and $\alpha_1,\ldots, \alpha_m\in k$ then the expression $\psi(\alpha_1,\ldots, \alpha_m)$ we get by substituting $X_i= \alpha_i$ for $1\leq i\leq m$ is either true or false.

We can turn a formula into a sentence by quantifying over the free variables: e.g., quantifying over all $X$ in the third formula above gives the sentence $(\forall X)\,X\neq 0 \rightarrow X^2= X$, which is false for any field with more than 2 elements.  Note that if $k= {\mathbb R}$ then the expression $(\forall X)(\exists Y)\, Y^2= \pi X$ is not a sentence in the above sense as it involves the real parameter $\pi$.  This problem does not arise with the expression $(\forall X) (\exists Y) (\exists Z)\,X= Y^2+ Z^2+ 2$: we don't need to regard 2 as a real parameter, because we get $2= 1+1$ for free by adding the constant symbol 1 to itself.  More generally, if $f(X_1,\ldots, X_m)$ is a polynomial over ${\mathbb Z}$ in variables $X_1,\ldots, X_m$ then (say) the expression $(\exists X_1)\cdots (\exists X_m)\, f(X_1,\ldots, X_m)= 0$ is a sentence.

An infinite field $k$ is {\em pseudo-finite} if it is a model of the theory of finite fields: that is, a sentence is true for $k$ if it is true for every finite field.  A nonprincipal ultraproduct of an infinite collection of finite fields is pseudo-finite.  For instance, if ${\mathcal U}$ is a nonprincipal ultrafilter on the set of all primes then the ultraproduct $\prod_{\mathcal U} {\mathbb F}_p$ is a pseudo-finite field of characteristic 0, and many of the infinite subfields of $\ovl{{\mathbb F}_p}$ are pseudo-finite fields of characteristic $p$.  See \cite{chatzidakis} and \cite[(5.1)]{chatzidakis2} for more details and examples.

We will use the following version of the Lefschetz Principle (see \cite[Cor.\ 2.2.9 and Cor.\ 2.2.10]{marker}).

\begin{thm}[The Lefschetz Principle]
\label{thm:Lefschetz}
 Let $\psi$ be a sentence.
 
 (a) Let $p$ be 0 or a prime.  If $\psi$ is true for some algebraically closed field of characteristic $p$ then $\psi$ is true for every algebraically closed field of characteristic $p$.
 
 (b) $\psi$ is true for some algebraically closed field of characteristic 0 if and only if for all but finitely many primes $p$, $\psi$ is true for some algebraically closed field of characteristic $p$.
\end{thm}

\noindent We have an immediate corollary.

\begin{cor}
\label{cor:Fpbar}
 Let $\psi$ be a sentence.  If $\psi$ is true for $k= \ovl{{\mathbb F}_p}$ for every prime $p$ then $\psi$ is true for every algebraically closed field.
\end{cor}

Now fix a field $k$.  Let ${\mathcal G}$ be a linear algebraic group over $k$.  Choose an embedding of ${\mathcal G}$ as a closed subgroup of ${\rm SL}_n$ for some $n$.  We regard ${\rm SL}_n(k)$ as a subset of affine space $k^{n^2}$ in the usual way.  Denote the co-ordinates of $k^{n^2}$ by $X_{ij}$ for $1\leq i\leq n$ and $1\leq j\leq n$.  Let $G= {\mathcal G}(k)$.  Our embedding of ${\mathcal G}$ in ${\rm SL}_n$ allows us to regard $G$ as a subset of $k^{n^2}$ given by the set of zeroes of some polynomials $f_1(X_{ij}),\ldots, f_r(X_{ij})$ over $k$ in the $X_{ij}$.

Fix $s,t\in {\mathbb Z}$.  We want to interpret the condition in $(*)_{s,t}$ as a sentence.  The subset ${\rm SL}_n(k)$ of $k^{n^2}$ is the set of zeroes of finitely many polynomials over ${\mathbb Z}$ in the matrix entries and the group operations on ${\rm SL}_n(k)$ are given by polynomials over ${\mathbb Z}$ in the matrix entries, so the conditions on $\alpha, \beta, \gamma, \delta, \sigma, \tau\in {\rm SL}_n(k)$ that $\delta= \gamma^s$, $\gamma^t= \delta$, $\gamma= [\alpha, \beta]$ and $\delta= [\sigma, \tau]$ are given by formulas in the matrix entries of $\alpha, \beta, \gamma, \delta, \sigma$ and $\tau$.  (For instance, if we denote the co-ordinates of $\alpha$ by $X_{ij}$ then the condition $\alpha\in {\rm SL}_n(k)$ is given by the formula ${\rm det}(X_{ij})- 1= 0$.)  However, the conditions $\alpha, \beta, \gamma, \delta, \sigma, \tau\in G$ may fail to be given by formulas as the polynomials $f_l(X_{ij})$ may involve some arbitrary elements of $k$.

We avoid this problem as follows.  Fix $r\geq n$, let $m_1(X_{ij}),\ldots, m_c(X_{ij})$ be a listing (in some fixed but arbitrary order) of all the monomials in the $X_{ij}$ of total degree at most $r$, and let $V_r$ be the subspace of the polynomial ring $k[X_{ij}]$ spanned by the $m_e(X_{ij})$.  Let $Z_{ab}$ for $1\leq a\leq c$ and $1\leq b\leq c$ be variables.  Define $g_d(X_{ij}, Z_{ab})= \sum_{e= 1}^c Z_{ed}m_e(X_{ij})$ for $1\leq d\leq c$.  Now let $\epsilon= (\epsilon_{ab})_{1\leq a\leq c, 1\leq b\leq c}$ be a tuple of elements of $k$.  We define
$$ G_\epsilon= \{\eta_{ij}\in k^{n^2}\,|\,{\rm det}(\eta_{ij})= 1\ \mbox{and}\ g_d(\eta_{ij}, \epsilon_{ab})= 0\ \mbox{for $1\leq d\leq c$}\}.  $$
If $H$ is a subgroup of ${\rm SL}_n(k)$ defined by polynomials over $k$ in the $X_{ij}$ of degree at most $r$ then we say $H$ {\em has complexity at most $r$}: in particular, $G_\epsilon$ has complexity at most $r$.\footnote{We assume that $r\geq n$, so the polynomial ${\rm det}(\beta_{ij})- 1$ has degree $n\leq r$.}  Conversely, any closed subgroup $H$ of complexity at most $r$ is of the form $G_\epsilon$ for some $\epsilon$ (note that we need at most $c$ polynomials of the form $g_d$ to define a subgroup of complexity at most $r$ because $\dim V_r= c$).  Any closed subgroup of ${\rm SL}_n(k)$ has complexity at most $r'$ for some $r'\geq n$.

Let $\phi(Z_{ab})$ be the formula in free variables $Z_{ab}$ given by
$$ (\forall U_{ij})(\forall V_{ij}) [[{\rm det}(U_{ij})= 1\wedge {\rm det}(V_{ij})= 1]\wedge [g_d(U_{ij}, Z_{ab})= 0\ \wedge\ g_d(V_{ij}, Z_{ab})= 0 \ \mbox{for $1\leq d\leq c$}]] $$
$$ \rightarrow g_d(W_{ij}, Z_{ab})= 0 \ \mbox{for $1\leq d\leq c$}, $$
where $W_{ij}$ is shorthand for $\sum_{l= 1}^n U_{il}V_{lj}$.  Then $G_\epsilon$ is closed under multiplication if and only if $\phi(\epsilon_{ab})$ is true.  Likewise, there are formulas $\chi(Z_{ab})$ and $\eta(Z_{ab})$ such that $G_\epsilon$ is closed under taking inverses if and only if $\chi(\epsilon_{ab})$ is true, and the identity $I$ belongs to $G_\epsilon$ if and only if $\eta(\epsilon_{ab})$ is true.

Now consider the condition $\psi_{n,s,t,r}$ given by
$$ [(\forall Z_{ab})\ \phi(Z_{ab})\wedge \chi(Z_{ab})\wedge \eta(Z_{ab})]\rightarrow $$
$$ (\forall \alpha\in G_Z)(\forall \beta\in G_Z)(\forall \delta\in G_Z)\ (\delta= [\alpha, \beta]^s\ \wedge\ [\alpha, \beta]= \delta^t) \rightarrow ((\exists \sigma\in G_Z)(\exists \tau\in G_Z)\,\delta= [\sigma, \tau]), $$

\noindent where $\alpha= (\alpha_{ij}), \beta= (\beta_{ij}), \dots$ are tuples representing elements of ${\rm SL}_n(k)$.  Here $\alpha\in G_Z$ is shorthand for $[{\rm det}(\alpha_{ij})= 1]\wedge \left[\bigwedge_{1\leq d\leq c} g_d(\alpha_{ij}, Z_{ab})= 0\right]$, and likewise for $\beta\in G_Z$, etc.  We regard $\psi_{n,s,t,r}$ as a sentence in variables $Z_{ab}$, $\alpha_{ij}$, $\beta_{ij}$, etc.  The above discussion yields the following: for any field $k$ and for any $n,s,t,r$, the sentence $\psi_{n,s,t,r}$ is true for $k$ if and only if\medskip\\
$(\dagger)$\ for every closed subgroup $G$ of ${\rm SL}_n(k)$ of complexity at most $r$ and for every $\alpha, \beta, \delta\in G$ such that $\delta= [\alpha, \beta]^s$ and $[\alpha, \beta]= \delta^t$, there exist $\sigma, \tau\in G$ such that $\delta= [\sigma, \tau]$.

We see from the above discussion that
$$ (\flat)\ \ \ \ \ \mbox{$\mathcal{G}(k)$ has the Honda property for every linear algebraic group $\mathcal{G}$ over $k$ if and only if} $$
$$ \mbox {the sentence $\psi_{n,s,t,r}$ is true for $k$ for all $n,s,t,r$.} $$

\begin{prop}
\label{prop:allpsi}
 Let $n\in {\mathbb N}$, let $r\geq n$ and let $s,t\in {\mathbb Z}$.  Let $k$ be an algebraically closed field or a pseudo-finite field.  Then $\psi_{n,s,t,r}$ is true for $k$.
\end{prop}

\begin{proof}
 Fix $n, s, t$ and $r$.  If $k= \ovl{{\mathbb F}_p}$ for some prime $p$ then $\mathcal{G}(k)$ has the Honda property for every linear algebraic group $\mathcal{G}$ over $k$ by Remark~\ref{rem:strong}, so $\psi_{n,s,t,r}$ is true for $k$ by $(\flat)$.  Hence $\psi_{n,s,t,r}$ is true for any algebraically closed field $k$ by Corollary~\ref{cor:Fpbar}.  If $k$ is a finite field then $\mathcal{G}(k)$ has the Honda property for every linear algebraic group $\mathcal{G}$ over $k$ by Theorem~\ref{thm:finite}, so $\psi_{n,s,t,r}$ is true for $k$ by $(\flat)$.  Hence $\psi_{n,s,t,r}$ is true for any pseudo-finite field.
\end{proof}

\begin{proof}[Proof of Theorem~\ref{thm:algebraic}(a) and (b)]
 This follows immediately from Proposition~\ref{prop:allpsi} and $(\flat)$.
\end{proof}

\begin{rem}
 We don't know whether $\mathcal{G}(k)$ must satisfy the Honda property for $\mathcal{G}$ a linear algebraic group over an arbitrary field $k$.  Tiersma has observed that if the conclusion of Theorem~\ref{thm:algebraic} holds for a field $k$ then it holds for any algebraic field extension $k'$ of $k$.  To see this, let $\mathcal{G}'$ be a linear algebraic group over $k'$.  Let $\alpha, \beta\in \mathcal{G}'(k')$.  Let $\gamma= [\alpha, \beta]$ and let $\delta\in \mathcal{G}'(k')$ such that $\langle \delta\rangle= \langle \gamma\rangle$.  There is a field $k_1$ such that $k\subseteq k_1\subseteq k'$, $k_1/k$ is finite, $\mathcal{G}'$ descends to a $k_1$-group $\mathcal{G}_1$ and $\alpha, \beta, \gamma, \delta\in \mathcal{G}_1(k_1)$.  Let $\mathcal{G}$ be the Weil restriction $R_{k_1/k}(\mathcal{G}_1)$ \cite[A.5]{CGP}.  Then $\mathcal{G}(k)= \mathcal{G}_1(k_1)$.  Since $\mathcal{G}(k)$ satisfies the Honda property by assumption, the result follows.
 
 Here is another situation where we obtain a positive result.  Let $G$ be a connected compact Lie group.  Then every element of $[G,G]$ is a commutator \cite[Thm.\ 6.55]{HM}, and it follows that $G$ has the Honda property.  The link with linear algebraic groups is the following: if ${\mathcal G}$ is a Zariski-connected real reductive algebraic group then $G:= {\mathcal G}({\mathbb R})$ is compact if and only if ${\mathcal G}$ is anisotropic, and in this case $G$ is connected in the standard topology \cite[V.24.6(c)(ii)]{borel}.  A similar argument using \cite[Theorem]{ree} shows that any connected reductive linear algebraic group over an algebraically closed field has the Honda property; this recovers a special case of Theorem~\ref{thm:algebraic}(a).
\end{rem}

\section{Profinite groups}
\label{sec:profinite}

In this section we prove Theorem~\ref{thm:algebraic}(c); by a nonarchimedean local field we mean either a finite extension of a $p$-adic field or the field of Laurent series ${\mathbb F}_q(\!(T)\!)$ in an indeterminate $T$ for some prime power $q$.  In fact, we prove a more general result for profinite groups.  Recall that a topological group is profinite if and only if it is compact and Hausdorff and has a neighbourhood base at the identity consisting of open subgroups \cite[Prop.\ 1.2]{DDMS}.

\begin{prop}
\label{prop:profinite}
Every profinite group has the Honda property.
\end{prop}

\begin{proof}
 Let $G$ be a profinite group.  Let $Q$ be the set of finite-index normal subgroups of $G$.  Let $\gamma, \delta\in G$ such that $\gamma$ is a commutator and $\langle \delta\rangle= \langle \gamma\rangle$.   For $N\in Q$, let $\pi_N\colon G\to G/N$ denote the canonical projection, and set $C_N= \{(\sigma, \tau)\in G\times G\,|\,\pi_N(\delta)= [\pi_N(\sigma), \pi_N(\tau)]\}$, a closed subset of $G\times G$.  By Theorem~\ref{thm:finite}, $C_N$ is nonempty for each $N\in Q$.  If $N_1,\ldots, N_t\in Q$ then $N_1\cap\cdots \cap N_t\in Q$ and $C_{N_1\cap\cdots \cap N_t}= C_{N_1}\cap\cdots \cap C_{N_t}$; in particular, $C_{N_1}\cap\cdots \cap C_{N_t}$ is nonempty.  By the finite intersection property for compact spaces, $\bigcap_{N\in Q} C_N$ is nonempty, so we can pick an element $(\sigma, \tau)$ from it.  Then $\pi_N(\delta)= [\pi_N(\sigma), \pi_N(\tau)]$ for all $N\in Q$.  But $\bigcap_{N\in Q} N= 1$ as $G$ is profinite, so $[\sigma, \tau]= \delta$ and we are done.
\end{proof}

\begin{rem}
 Lenstra has observed that if $G$ is profinite then a variation of the Honda property holds: if $\gamma, \delta\in G$, $\gamma$ is a commutator and the closures of $\langle \gamma\rangle$ and $\langle \delta\rangle$ are equal then $\delta$ is also a commutator.  The argument is very similar to the proof of Proposition~\ref{prop:profinite}.
\end{rem}

\begin{proof}[Proof of Theorem~\ref{thm:algebraic}(c)]
 Let $F$ be a nonarchimedean local field with associated norm $\nu$ and valuation ring $k$.  Fix a uniformiser $\pi$.  Let ${\mathcal G}$ be a linear algebraic group over $k$ and let $G= {\mathcal G}(k)$.  It is enough by Proposition~\ref{prop:profinite} to show that $G$ is profinite.  By assumption, there is an embedding of ${\mathcal G}$ as a closed subgroup of ${\rm SL}_n$ for some $n$.  So $G$ is a subgroup of ${\rm SL}_n(k)$ and there are polynomials $f_1,\ldots, f_t$ in $n^2$ variables over $k$ for some $t\in {\mathbb N}_0$ such that $G= \{a_{ij}\in k^{n^2}\,|\,f_i(a_{ij})= 0\ \mbox{for $1\leq i\leq t$}\}$ (we need only finitely many polynomials as DVRs are noetherian).  The operations of addition and multiplication on $k$ are continuous; it follows easily that $G$ is a closed subspace of $k^{n^2}$ with respect to the topology on $k^{n^2}$ induced by $\nu$, and the group operations on $G$ are continuous.  Hence $G$ is a compact topological group.  The open subgroups $G_n:= \{g\in G\,|\,g= 1 \mod \pi^n\}$ form a neighbourhood base at the identity.  It follows that $G$ is profinite, as required.
\end{proof}


\bigskip
\bigskip
\noindent {\bf Acknowledgements:} I'm grateful to Hendrik Lenstra for introducing me to the Honda property, for stimulating discussions and for sharing a draft of his preprint \cite{lenstra} with me.  I'm grateful to Hendrik and to Samuel Tiersma for comments on earlier drafts of this note.  I'd also like to thank the referees for some comments and corrections, and the Edinburgh Mathematical Society for supporting Lenstra's visit to Aberdeen.

\bigskip
\bigskip
\noindent {\bf Competing interests:} The author declares none.

\end{document}